\documentclass[12pt]{amsart}

\usepackage{amsmath,amssymb,amsfonts,amssymb,amsthm}

\usepackage{verbatim}
\usepackage[usenames]{color}
\usepackage{hyperref}
\usepackage{url}
\usepackage{tikz,tikz-qtree,ifthen,cancel}
\usepackage{array,tikz-qtree,ifthen,cancel}
\usepackage{graphicx}
\usepackage{adjustbox}
\usepackage{amsthm,graphicx,tikz,appendix,tikz-qtree,ifthen,cancel, enumitem}
\usetikzlibrary{calc,shapes,patterns,positioning}

\usepackage{yfonts}

\usepackage{accents}

\usepackage{blindtext}
\usepackage{setspace}

\DeclareFontFamily{U}{mathb}{\hyphenchar\font45}
\DeclareFontShape{U}{mathb}{m}{n}{
      <5> <6> <7> <8> <9> <10> gen * mathb
      <10.95> mathb10 <12> <14.4> <17.28> <20.74> <24.88> mathb12
}{}
\DeclareSymbolFont{mathb}{U}{mathb}{m}{n}
\DeclareMathSymbol{\llcurly}{3}{mathb}{"CE}
\DeclareMathSymbol{\ggcurly}{3}{mathb}{"CF}
\DeclareFontFamily{U}{matha}{\hyphenchar\font45}
\DeclareFontShape{U}{matha}{m}{n}{
      <5> <6> <7> <8> <9> <10> gen * matha
      <10.95> matha10 <12> <14.4> <17.28> <20.74> <24.88> matha12
      }{}
\DeclareSymbolFont{matha}{U}{matha}{m}{n}
\DeclareMathSymbol{\curlywedge} {2}{matha}{"4E}
\DeclareMathSymbol{\curlyvee} {2}{matha}{"4F}

\newtheorem{thm}{Theorem}[section]
\newtheorem*{thmunnum}{Theorem}
\newtheorem{prop}[thm]{Proposition}
\newtheorem*{thmmain}{Main Theorem}

\newtheorem{lem}[thm]{Lemma}
\newtheorem{cor}[thm]{Corollary}

\newtheorem*{thm7.3}{Theorem \ref{thm.MillikenSWP}}

\newtheorem{fact}[thm]{Fact}

\theoremstyle{remark}
\newtheorem{rem}[thm]{Remark}

\theoremstyle{definition}
\newtheorem{defn}[thm]{Definition}

\newtheorem{question}[thm]{Question}

\theoremstyle{remark}

\newcommand{\A}{\mathrm{A}}

\newcommand{\noprint}[1]{\relax}

\begin{document}

\begin{center}
{\Large\sc On the number of $Q$-points}\\[1.8ex]


{\small Silvan Horvath}\\[1.2ex] 
{\scriptsize Department of Mathematics, ETH Z\"urich, 
8092 Z\"urich, Switzerland\\ 
silvan.horvath@math.ethz.ch}\\[1.8ex]

{\small Tan {\"O}zalp}\\[1.2ex] 
{\scriptsize Department of Mathematics, University of Notre Dame,
Notre Dame, IN 46556, USA\\
aozalp@nd.edu}\\[1.8ex]

\end{center}


\begin{quote}
{\small {\bf Abstract.} We show that, up to isomorphism, the number of $Q$-points is either finite, $2^{\mathfrak{d}}$ or $2^{\mathfrak{c}}$. This answers a question asked by Borodulin--Nadzieja, Mart{\'\i}nez-Celis, Morawski and {\'S}wierczy{\'n}ska \cite{bdmcms-measures}, and by Halbeisen and the authors \cite{halhoroza-qpoints}. We also show that under mild hypotheses, the existence of infinitely many $Q$-points implies the existence of non-atomic $Q$-measures, and of $2^{\mathfrak{c}}$-many Tukey-top $Q$-points, strengthening results of Raghavan \cite{raghavan2025q} and of Borodulin--Nadzieja et al. \cite{bdmcms-measures}.
}
\end{quote}

\begin{quote}
\end{quote}


\setcounter{section}{0}
\section{Introduction}

The primary purpose of this paper is to answer a question asked by Borodulin--Nadzieja, Mart{\'\i}nez-Celis, Morawski and {\'S}wierczy{\'n}ska \cite{bdmcms-measures}, as well as by Halbeisen and the authors \cite{halhoroza-qpoints}, of whether the existence of exactly $\aleph_0$ many $Q$-points is consistent. The number of $Q$-points -- or the number of any other type of special ultrafilter mentioned in this paper -- always refers to the number of Rudin-Keisler isomorphism classes of this type, i.e., modulo permutations of $\omega$. 

While it is a classical result of Miller \cite{miller1980there} that the existence of $Q$-points is independent of \textsf{ZFC}, the more general question above has received less attention. The analogous question for $P$-points and Ramsey ultrafilters has been investigated by Shelah \cite[Ch.\,VI,\S5]{she}, who showed that there are models with exactly $\kappa$-many $P$-points and Ramsey ultrafilters for any $0 \leq \kappa \leq \aleph_2$ (also, see Chodounsk\'y \cite{chodounsky2012katowice}). In contrast, our results imply that the number of $Q$-points is neither $\aleph_0$ nor $\aleph_1$. For finite $\kappa$, the existence of exactly $\kappa$ $Q$-points is consistent, as was shown by the authors and Halbeisen in \cite{halhoroza-qpoints}.

We show that the number of $Q$-points is always either finite or equal to one of the cardinals $2^{\mathfrak{d}}$ or $2^{\mathfrak{c}}$. This strong restriction might be unsurprising, since the number of $Q$-points is related to the number of \emph{near coherence classes of ultrafilters}, $Q$-points being isomorphic if and only if they are nearly coherent. By results due to Banakh and Blass \cite{banakh2006number}, the number of near coherence classes is always either finite or $2^{\mathfrak{c}}$. Indeed, most of the arguments that follow are already found in that paper, and we leave open the question whether the dichotomy governing the number of near coherence classes in fact also holds for the number of $Q$-points, i.e., whether there actually are models with  exactly $2^{\mathfrak{d}}$-many $Q$-points and $2^{\mathfrak{d}}<2^{\mathfrak{c}}$.

Unbeknownst to us until recently,  Mill\'an \cite{millan2020filters} independently proved that the number of $Q$-points does not lie in the interval $(2^\mathfrak{d}, 2^\mathfrak{c})$. We prove the following.
\vspace{1ex}
\begin{thmmain}~
\begin{enumerate}
\item If $\mathfrak{u}>\mathfrak{d}$, then there either exists no $Q$-point, or there exists a closed set of $Q$-points of size $2^{\mathfrak{c}}$.
\item If $\mathfrak{u}=\mathfrak{d}$, then the number of $Q$-points is either finite or $2^{\mathfrak{d}}$, or there exists a closed set of $Q$-points of size $2^{\mathfrak{c}}$.
\item If $\mathfrak{u}<\mathfrak{d}$, then the number of $Q$-points is either finite, or there exists a closed set of $Q$-points of size $2^{\mathfrak{c}}$.
\item If there is an $\omega^{\ast}$-discrete ultrafilter of character $>\mathfrak{d}$, then the number of $Q$-points is either finite, or there exists a closed set of $Q$-points of size $2^{\mathfrak{c}}$.
\item If there are infinitely many Ramsey ultrafilters, then there exists a closed set of $Q$-points of size $2^{\mathfrak{c}}$.
\end{enumerate}
\end{thmmain}

The following is a well-known result of Mill\'an \cite{millan2007note} that complements our main theorem.

\vspace{1ex}
\begin{thmunnum}[\cite{millan2007note}]~
\begin{enumerate}
\setcounter{enumi}{5}
\item If $\operatorname{cov}(\mathcal{M})=\mathfrak{d}$, then there exists a closed set of $Q$-points of size $2^{\mathfrak{c}}$.
\end{enumerate}
\end{thmunnum}

These infinite closed sets of $Q$-points we construct will be the set of ultrafilters extending some given $Q$-filter. In addition to proving the Main Theorem, we also strengthen a result of Borodulin--Nadzieja, Mart{\'\i}nez-Celis, Morawski and {\'S}wierczy{\'n}ska \cite{bdmcms-measures} and show that the existence of a $Q$-filter that extends to infinitely many ultrafilters is equivalent to the existence of a \emph{non-atomic $Q$-measure}. Furthermore, we observe that the existence of an infinite closed set of $Q$-points implies the existence of $2^{\mathfrak{c}}$-many pairwise $\leq_{\operatorname{RK}}$-incomparable Tukey-top $Q$-points, which strengthens results of Raghavan \cite{raghavan2025q}.
\section{Preliminaries}

\begin{defn}
A filter $\mathcal{F}$ is a \emph{$Q$-filter} if for every partition $I$ of $\omega$ into intervals, there exists $A \in \mathcal{F}$ such that $A$ intersects each interval in $I$ in at most one point. We call such an $A$ a \emph{selector} for $I$. A $Q$-filter that is an ultrafilter is called a \emph{$Q$-point}..	
\end{defn}

We denote by $\omega^{\ast}$ the Stone-\v{C}ech remainder of $\omega$, i.e., the set of non-principal ultrafilters on $\omega$, equipped with the topology generated by clopen sets of the form $\{\mathcal{U}\in \omega^{\ast}: A \in \mathcal{U}\}$ for $A \in [\omega]^\omega$. Of particular importance in this paper are countably infinite discrete sequences in $\omega^{\ast}$. Note that a sequence  of ultrafilters $\langle \mathcal{U}_n: n \in \omega\rangle$ is discrete if and only if there exists a sequence $\langle A_n: n \in \omega\rangle$ of pairwise disjoint subsets of $\omega$ with $\forall n \in \omega: A_n \in \mathcal{U}_n$.

Given some ultrafilter $\mathcal{V}$ and a sequence $\langle\mathcal{U}_n: n \in \omega\rangle$ in $\omega^{\ast}$, the ultrafilter $\mathcal{V}-\lim_n \mathcal{U}_n$ consists of those $A \subseteq \omega$ for which
\[\{n \in \omega: A \in \mathcal{U}_n\}\in \mathcal{V}.\]
Observe that for any $f\in {^\omega}\omega$ we have $f(\mathcal{V}-\lim_n \mathcal{U}_n)=\mathcal{V}-\lim_n f(\mathcal{U}_n)$, where $f(\mathcal{U})$ is the ultrafilter generated by $\{f[A]: A \in \mathcal{U}\}$.

The following is a result of M.~E.~Rudin \cite{rudin1965types}.

\begin{lem}\label{lem:rudin}
Let $\langle \mathcal{U}_n: n \in \omega\rangle$ be a sequence of ultrafilters and assume that $\mathcal{V}$ and $\mathcal{V}'$ are distinct ultrafilters such that there exists $A \in \mathcal{V}\cap \mathcal{V'}$ for which $\langle \mathcal{U}_n: n \in A\rangle$ is discrete. Then, \[\mathcal{V}-\lim_n \mathcal{U}_n \neq \mathcal{V}'-\lim_n \mathcal{U}_n.\]
\end{lem}

Apart from $Q$-points, the following types of special ultrafilters appear in our paper.

\begin{defn}
Let $\mathcal{U}$ be an ultrafilter on $\omega$.
\begin{itemize}
\item $\mathcal{U}$ is a \emph{$P$-point} if for every countable set $\{A_n: n \in \omega\}$ of elements of $\mathcal{U}$, there exists $B \in \mathcal{U}$ with $\forall n \in \omega: B \subseteq^{*}A_n$.
\item $\mathcal{U}$ is a \emph{Ramsey ultrafilter} if for every partition $\mathcal{P}$ of $\omega$ with $\mathcal{P} \cap \mathcal{U}=\varnothing$, $\mathcal{U}$ contains a selector for $\mathcal{P}$. Equivalently, $\mathcal{U}$ is both a $Q$-point and a $P$-point.
\end{itemize}

Additionally, we require the following type of ultrafilter introduced by Banakh and Blass~\cite{banakh2006number}.
\begin{itemize}
\item $\mathcal{U}$ is $\omega^{\ast}$-discrete if for every countable sequence $\langle \mathcal{V}_n: n \in \omega \rangle$ of pairwise distinct ultrafilters, there exists $A \in \mathcal{U}$ such that $\langle \mathcal{V}_n: n \in A\rangle$ is discrete.
\end{itemize}
\end{defn}

The \emph{character} $\chi(\mathcal{F})$ of a filter $\mathcal{F}$ is the least size of a basis $\mathcal{B}$ of $\mathcal{F}$, i.e., the least size of a family $\mathcal{B}\subseteq\mathcal{F}$ such that for all $A\in\mathcal{F}$, there is $B\in\mathcal{B}$ with $B\subseteq A$. The cardinal characteristic $\mathfrak{u}$ is defined as
\[\mathfrak{u}=\min\{\chi(\mathcal{U}): \mathcal{U}\in \omega^{\ast}\}.\]

The second cardinal characteristic we need is the \emph{dominating number} of a filter $\mathcal{F}$.

\begin{defn}
Let $\mathcal{F}$ be a filter. The cardinal $\mathfrak{d}(\mathcal{F})$ denotes the least size of a family $\mathcal{D}\subseteq {^\omega}\omega$ such that for every $f\in {^\omega}\omega $, there is $g\in\mathcal{D}$ with $f\leq_{\mathcal{F}}g$. Here, $f\leq_{\mathcal{F}}g$ means $\{n : f(n)\leq g(n)\}\in\mathcal{F}$. If $\mathcal{F}$ is the Fr\'echet filter, then $\mathfrak{d}(\mathcal{F})=:\mathfrak{d}$.
\end{defn}

The following is a corollary of Proposition 19 and Proposition 20 of \cite{banakh2006number}.

\begin{prop}\label{prop:banakhblassinfiniteint}
Let $\{\mathcal{U}_n: n\in\omega\}$ be a set of ultrafilters and $\mathcal{W}$ a $P$-point not nearly coherent to any of the $\mathcal{U}_n$. Then $\mathfrak{d}(\bigcap_{n\in\omega}\mathcal{U}_n)\leq\chi(\mathcal{W})$.
\end{prop}

It is well-known that any ultrafilter $\mathcal{U}$ with $\chi(\mathcal{U})<\mathfrak{d}$ is a $P$-point (see Ketonen \cite{ketonen1976existence}). Moreover, for any ultrafilter $\mathcal{U}$, the ultrafilter $\mathcal{U}^2$ is never a $P$-point.\footnote{Here, $\mathcal{U}^2:=\mathcal{U}-\lim_{n}\mathcal{V}_n$, where $\{\mathcal{V}_n: n\in\omega\}$ is any discrete subset of $\omega^*$ such that each $\mathcal{V}_n$ is isomorphic to $\mathcal{U}$.} Therefore, by Corollary 2.9 of Baumgartner \cite{baumgartner-uf} and Proposition 12 of Banakh and Blass \cite{banakh2006number}, we obtain the following (see also Lemma 11 of \cite{banakh2006number}).
\begin{lem}\label{lem:baumgartner}
Let $\mathcal{U}$ be a $P$-point. Then both $\mathcal{U}$ and $\mathcal{U}^2$ are $\omega^*$-discrete ultrafilters and $\chi(\mathcal{U}^2)\geq \mathfrak{d}$.
\end{lem}

We say that two filters $\mathcal{F}$ and $\mathcal{G}$ are \emph{nearly coherent} if there exists some finite-to-one function $f \in {^\omega}\omega$ such that $f(\mathcal{F}) \cup f(\mathcal{G})$ generates a filter. Restricted to the set of ultrafilters, near coherence is an equivalence relation (see \cite{blass86nearcoherence}) and its equivalence classes are called \emph{near coherence classes}. It is well-known and straightforward to verify that two $Q$-points are nearly coherent if and only if they are isomorphic.

The following is a consequence of the fact that $|\omega^\omega|=\mathfrak{c}$. Note that $2^{\mathfrak{c}}$ is cardinality of the entire space $\omega^{\ast}$.
\begin{fact}\label{fact:numberofiso}
Let $\mathcal{C}$ be a set of ultrafilters with $|\mathcal{C}|=2^\mathfrak{c}$. Then there exists $\mathcal{C}'\in[\mathcal{C}]^{2^{\mathfrak{c}}}$ consisting of pairwise non-isomorphic ultrafilters.
\end{fact}

For a filter $\mathcal{F}$, we denote by $[\mathcal{F}]$ the set of ultrafilters extending $\mathcal{F}$. Note that every nonempty closed subset of $\omega^{\ast}$ is of the form $[\mathcal{F}]$ for some filter $\mathcal{F}$. We will need the following two propositions on the cardinality of $[\mathcal{F}]$. The first is a theorem of Posp\'i\v{s}il \cite{pospisil-ind}, and the second is Lemma 6 of Banakh and Blass \cite{banakh2006number}. We provide the proofs for completeness.

\begin{prop}[\cite{pospisil-ind}]\label{prop:pospisil}
Let $\mathcal{F}$ be a filter. If $[\mathcal{F}]$ is infinite, then it has cardinality $2^{\mathfrak{c}}$.
\end{prop}

\begin{proof}
Assume that there are infinitely many ultrafilters extending $\mathcal{F}$. Since any infinite set in a regular space has an infinite discrete subset, we may assume that these ultrafilters form a discrete set. Hence, we find pairwise disjoint sets $A_i \in [\omega]^{\omega}, i \in \omega$, such that each $A_i$ has infinite intersection with every member of $\mathcal{F}$. Let $\mathcal{I}=\{X_\alpha: \alpha \in \mathfrak{c}\}$ be an independent family of cardinality $\mathfrak{c}$. For each $\varphi \in {^\mathfrak{c}}2$, let 
\[\mathcal{F}_{\varphi}:=\mathcal{F}\cup \left\{\bigcup_{i \in X_\alpha}A_i: \varphi(\alpha)=0\right\}\cup \left\{\bigcup_{i \in \omega \setminus X_\alpha}A_i: \varphi(\alpha)=1\right\}.\]
Each $\mathcal{F}_{\varphi}$ can be extended to an ultrafilter $\mathcal{U}_{\varphi}$. By construction, these ultrafilters are pairwise distinct.
\end{proof}

\begin{prop}[\cite{banakh2006number}]\label{prop:banakhblasslemma6}
Let $\mathcal{F}$ be a filter and $\mathcal{U} \supseteq \mathcal{F}$ an ultrafilter such that $\chi(\mathcal{F})< \chi(\mathcal{U})$. Then $[\mathcal{F}]$ has cardinality $2^{\mathfrak{c}}$.
\end{prop}

\begin{proof}
Let $\mathcal{B}$ be a basis of $\mathcal{F}$ of cardinality $<\chi(\mathcal{U})$. By Lemma \ref{prop:pospisil}, it suffices to show that $[\mathcal{F}]$ is infinite. Hence, assume by contradiction that $\mathcal{F}$ only extends to finitely many ultrafilters $\mathcal{U}, \mathcal{V}_0, \ldots, \mathcal{V}_{n-1}$. Find $A \in \mathcal{U}$ such that $A$ is in none of the $\mathcal{V}_i$. Now, $\mathcal{U}$ is the only ultrafilter extending $\mathcal{F}\cup \{A\}$, which implies that $\mathcal{B}\cup \{A\}$ is a basis of $\mathcal{U}$ of cardinality $<\chi(\mathcal{U})$.
\end{proof}

Before concluding this section, let us review some notions and facts about filters. First, we state the Jalili-Naini--Talagrand theorem. As usual, we endow $\mathcal{P}(\omega)$ with the natural topology homeomorphic to the Cantor space.
\begin{thm}[\cite{talagrand}]\label{talagrand}
Let $\mathcal{F}$ be a filter on $\omega$. The following are equivalent.
\begin{enumerate}
    \item $\mathcal{F}$ has the Baire property.
    \item $\mathcal{F}$ is meager.
    \item There is an interval partition $I=\{I_n: n\in\omega\}$ such that for every $A\in\mathcal{F}$, $\{n: A\cap I_n=\varnothing\}$ is finite.
    \item $\{e_A : A\in\mathcal{F}\}\subseteq\omega^\omega$ is bounded, where $e_A$ denotes the increasing enumeration of $A$.
\end{enumerate}
\end{thm}

\pagebreak

\section{Test families, $Q$-test families and $Q$-filter-test families}

The purpose of this section is to adapt the notion of a test family testing near coherence to families of interval partitions that test the $Q$-filter-/$Q$-point property. The following definition is due to Banakh and Blass \cite{banakh2006number} (also, see Blass \cite{blass86nearcoherence}).

\begin{defn}[\cite{banakh2006number}]
	Let $\mathcal{F}$ be a filter. A family $\mathcal{T}\subseteq \omega^\omega$ of finite-to-one functions is a \emph{test family} over $\mathcal{F}$ if for all nearly coherent ultrafilters $\mathcal{U}, \mathcal{V} \in [\mathcal{F}]$, the near coherence of $\mathcal{U}$ and $\mathcal{V}$ is witnessed by some function in $\mathcal{T}$.
\end{defn}

\begin{prop}[Banakh and Blass \cite{banakh2006number}]
	Let $\mathcal{F}$ be a filter. There is a test family of cardinality $\mathfrak{d}(\mathcal{F})$ over $\mathcal{F}$.
\end{prop}
 
We adapt this notion as follows.
 
\begin{defn}
Let $\mathcal{F}$ be a filter and let $\mathcal{I}$ be a collection of interval partitions of $\omega$.
\begin{itemize}
	\item $\mathcal{I}$ is a \emph{$Q$-test family} over $\mathcal{F}$ if whenever $\mathcal{U}\in[\mathcal{F}]$ is such that $\mathcal{U}$ contains a selector for each $I\in\mathcal{I}$, then $\mathcal{U}$ is in fact a $Q$-point.
	\item $\mathcal{I}$ is a \emph{$Q$-filter-test family} over $\mathcal{F}$ if whenever $G \supseteq \mathcal{F}$ is a filter such that $\mathcal{G}$ contains a selector for each $I\in\mathcal{I}$, then $\mathcal{G}$ is in fact a $Q$-filter.
\end{itemize}
\end{defn}

\begin{lem}\label{lem:q-test}
Let $\mathcal{F}$ be a filter
\begin{enumerate}[label=(\roman*)]
	\item There is a $Q$-test family of cardinality $\mathfrak{d}(\mathcal{F})$ over $\mathcal{F}$.
	\item If $\mathcal{F}$ is non-meager, there is a $Q$-filter-test family over $\mathcal{F}$ of cardinality $\mathfrak{d}(\mathcal{F})$.
\end{enumerate}
\end{lem}

\begin{proof}

(i) In light of (ii), we may assume that $\mathcal{F}$ is meager. It is not hard to verify that for such an $\mathcal{F}$ we have $\mathfrak{d}(\mathcal{F})=\mathfrak{d}$. Hence, it suffices to show that there is a $Q$-test family of cardinality $\mathfrak{d}$ over the Fr\'echet filter. Let $\mathcal{D}$ be a dominating family of cardinality $\mathfrak{d}$ and define for each $f\in \mathcal{D}$ an interval partition $J^f=\{[j^f_n, j^f_{n+1}): n\in\omega\}$ such that, for all $n\in\omega$,
\[k\leq j^f_n\Rightarrow f(k)< j^f_{n+1}.\]
Note that whenever $k<k'<f(k)$, then $k$ and $k'$ lie either in the same interval of $J^f$ or in adjacent intervals. We claim that $\mathcal{J}:=\{J^f: f \in \mathcal{D}\}$ is as required. Hence, let $\mathcal{U}\in [\mathcal{F}]$ be such that $\mathcal{U}$ contains a selector for each $J^f$ and consider any interval partition $\{[l_n, l_{n+1}): n \in \omega\}$. Define the function $g$ be to map any $k \in [l_n, l_{n+1})$ to $l_{n+2}$. We find $f \in \mathcal{D}$ such that $g(k)\leq f(k)$ for all $k$ above some $n_0$, as well as a selector $S \in \mathcal{U}$ for $J^f$. Now, it is not hard to see that any $[l_n, l_{n+1})$ intersects $S \setminus n_0 \in \mathcal{U}$ in at most two points. Since $\mathcal{U}$ is an ultrafilter, we are done.

(ii) Let $\mathcal{D}\subseteq {^\omega}\omega$ witness $\mathfrak{d}(\mathcal{F})=|\mathcal{D}|$. Again define for each $f \in \mathcal{D}$ an interval partition $J^f$ so that whenever $k<k'<f(k)$, then $k$ and $k'$ lie either in the same or in adjacent intervals of $J^f$. Since $\mathcal{F}$ is non-meager, we find some $A^f\in\mathcal{F}$ that avoids infinitely many intervals of $J^f$. Let $I^f$ be an interval partition obtained by gluing the blocks of $J^f$ in such a way that each interval of $I^f$ has an interval of $J^f$ that is met by $A^f$ as an initial segment, and has an interval of $J^f$ not met by $A^f$ as a terminal segment. It follows that whenever $k<k'<f(k)$ and $k \in A^f$, then $k$ and $k'$ lie in the same interval of $I^f$. Let $\mathcal{I}=\{I^f : f\in\mathcal{D}\}$. We claim that $\mathcal{I}$ works.

Hence, let $\mathcal{G}\supseteq \mathcal{F}$ be any filter extending $\mathcal{F}$ that contains a selector for each $I^f$. For any given interval partition $\{[l_n, l_{n+1}): n \in \omega\}$, let $g$ again be the function that maps $k \in [l_n, l_{n+1})$ to $l_{n+2}$. Find $f\in\mathcal{D}$ such that $A:=\{k : g(k)\leq f(k)\}\in\mathcal{F}$ and let $S\in \mathcal{G}$ be a selector for $I^f$. Assume towards a contradiction that for some interval $[l_n, l_{n+1})\in P$, the set $\bar S:= S \cap A \cap A^{f}\in \mathcal{G}$ intersects $[l_n, l_{n+1})$ in two distinct points $k<k'$. Since $k \in A$, we have
\[k<k'<l_{n+2}=g(k)\leq f(k).\]
Since $k \in A^f$, it follows that $k$ and $k'$ lie in the same interval of $I^f$, contradicting the assumption that $S$ is a selector for $I^f$.
\end{proof}

\section{Infinite closed sets of $Q$-points}

Recall that the nonempty closed subset of $\omega^{\ast}$ are precisely the sets of the form $[\mathcal{F}]$ for some filter $\mathcal{F}$. By Fact \ref{fact:numberofiso} and Proposition \ref{prop:pospisil}, we have the following.

\begin{fact}
If $\mathcal{F}$ is a $Q$-filter and $[\mathcal{F}]$ is infinite, then $[\mathcal{F}]$ is an infinite closed set of $Q$-points. In particular, the number of $Q$-points is $2^{\mathfrak{c}}$ by Fact \ref{fact:numberofiso}.
\end{fact}

We show that under any of the assumptions
\begin{enumerate}[label=(\roman*)]
\item there is a $Q$-point of character $>\mathfrak{d}$,
\item $\mathfrak{u}< \mathfrak{d}$ and there are infinitely many $Q$-points,
\item there is an $\omega^{\ast}$-discrete ultrafilter of character $>\mathfrak{d}$ and there are infinitely many $Q$-points,
\item there are infinitely many Ramsey ultrafilters,
\end{enumerate}
there exists a $Q$-filter $\mathcal{F}$ such that $[\mathcal{F}]$ is infinite. Note that this then proves items (1), (3), (4), and (5) of the Main Theorem. Item (2) will follow from a slightly more ad hoc construction.

We begin by proving item (i). This theorem was independently proved by Mill\'an \cite{millan2020filters}.

\begin{thm}
Assume that there is a $Q$-point of character $>\mathfrak{d}$. Then there is a $Q$-filter $\mathcal{F}$ such that $[\mathcal{F}]$ is infinite.
\end{thm}

\begin{proof}
Let $\mathcal{U}$ be a $Q$-point of character $>\mathfrak{d}$ and $\mathcal{I}$ a $Q$-test family of interval partitions of cardinality $\mathfrak{d}$ over the Fr\'echet filter. For each $I\in\mathcal{I}$, find a selector $S^I\in\mathcal{U}$ and let $\mathcal{F}\subseteq\mathcal{U}$ be the filter generated by $\{S^I : I\in\mathcal{I}\}$. Note that $\mathcal{F}$ is a $Q$-filter, and also
\[\chi(\mathcal{F})\leq|\mathcal{I}|=\mathfrak{d}<\chi(\mathcal{U}).\]
Hence, $[\mathcal{F}]$ is infinite by Proposition \ref{prop:banakhblasslemma6}.
\end{proof}

As corollaries, we obtain the following.

\begin{cor}
If $\mathfrak{u}>\mathfrak{d}$ and there exists a $Q$-point, then there is a $Q$-filter $\mathcal{F}$ such that $[\mathcal{F}]$ is infinite.
\end{cor}

\begin{cor}\label{cor:pigeonhole}
Assume that the number of $Q$-points is strictly greater than $2^{\mathfrak{d}}$. Then there is a $Q$-filter $\mathcal{F}$ such that $[\mathcal{F}]$ is infinite.
\end{cor}

\begin{proof}
Note that there are only $(\mathfrak{c}^{\mathfrak{d}}=2^{\mathfrak{d}})$-many $\leq \mathfrak{d}$-generated ultrafilters. Hence, if there are more than $2^{\mathfrak{d}}$-many $Q$-points, then at least one of them must be $>\mathfrak{d}$-generated. 
\end{proof}

We continue by proving item (ii), using arguments analogous to those in the proof of Proposition 22 of \cite{banakh2006number}.

\begin{thm}\label{thm:ulessd}
Assume that $\mathfrak{u}<\mathfrak{d}$ and that there are infinitely many $Q$-points. Then there is a $Q$-filter $\mathcal{F}$ such that $[\mathcal{F}]$ is infinite. In fact, if $\{\mathcal{U}_n: n\in\omega\}$ is a set of pairwise non-isomorphic $Q$-points, then the closure of $\{\mathcal{U}_n: n\in\omega\}$ in $\omega^*$ contains such an infinite closed set $[\mathcal{F}]$ as a subset.
\end{thm}

\begin{proof}
Let $\langle \mathcal{U}_n: n\in\omega\rangle$ be a sequence of non-isomorphic $Q$-points. We may assume that $\langle \mathcal{U}_n: n\in\omega\rangle$ is discrete, by extracting such a subsequence. Let $\mathcal{W}$ be any ultrafilter with $\chi(\mathcal{W})=\mathfrak{u}<\mathfrak{d}$. It follows that $\mathcal{W}$ is a $P$-point, so by Lemma \ref{lem:baumgartner}, $\mathcal{W}^2$ is an $\omega^*$-discrete ultrafilter with $\mathfrak{d}\leq\chi(\mathcal{W}^2)$. Note that $\mathcal{W}$ and the $\mathcal{U}_n$ are pairwise not nearly coherent.\footnote{We may either assume this without loss of generality, by possibly removing one of the $\mathcal{U}_n$, or we notice that a $Q$-point cannot be nearly coherent to a $<\mathfrak{d}$-generated ultrafilter.} Define $\mathcal{H}:=\bigcap_{n\in\omega}\mathcal{U}_n$. It follows from Proposition \ref{prop:banakhblassinfiniteint} that
\[\mathfrak{d}(\mathcal{H})\leq \chi(\mathcal{W})=\mathfrak{u}<\mathfrak{d}\leq \chi(\mathcal{W}^2).\]

Also, since $\mathcal{H}$ is the intersection of countably many ultrafilters, it is non-meager. Therefore, by Lemma \ref{lem:q-test} (ii), we find a $Q$-filter-test family of interval partitions for $\mathcal{H}$ of size $\mathfrak{u}$, say $\mathcal{I}=\{I_{\alpha}: \alpha < \mathfrak{u}\}$. We now construct a filter $\mathcal{G}$ by a recursive construction of length $\mathfrak{u}$. At a given stage $\alpha<\mathfrak{u}$, $\mathcal{G}_{\alpha}$ will be generated by $<\mathfrak{u}$-many sets, and we will also have $\mathcal{G}_{\alpha}\subseteq \mathcal{W}^2$.

Let $\mathcal{G}_0$ be the Fr\'echet filter. At limit stages, we take the union of the previously constructed filters. Now assume that we have constructed $\mathcal{G}_{\alpha}$. Let $f_{\alpha}:\omega\to\omega$ be the function that is constantly equal to $n$ on the $n$'th interval of $I_\alpha$. Since $\mathcal{W}^2$ is $\omega^*$-discrete and the ultrafilters $f_\alpha(\mathcal{U}_n)$ are pairwise distinct, we can find $A_{\alpha}\in \mathcal{W}^2$ such that $\langle f_{\alpha}(\mathcal{U}_n): n\in A_{\alpha}\}$ is discrete. Define $\mathcal{G}_{\alpha+1}$ to be the filter generated by $\mathcal{G}_{\alpha}$ and $A_{\alpha}$. At the end, set $\mathcal{G}=\bigcup_{\alpha<\mathfrak{u}} \mathcal{G}_{\alpha}$.

Note that $\mathcal{G}\subseteq \mathcal{W}^2$ and $\chi(\mathcal{G})<\chi(\mathcal{W}^2)$. It follows from Proposition \ref{prop:banakhblasslemma6} that $[\mathcal{G}]$ is infinite. Consider the filter
\[\mathcal{F}:=\mathcal{G}-\lim_n \mathcal{U}_n,\]
i.e., the filter consisting of those $A\in [\omega]^\omega$ for which $\{n \in \omega: A \in \mathcal{U}_n\}\in \mathcal{G}$. Note that for any $\mathcal{V}\in [\mathcal{G}]$, the ultrafilter $\mathcal{V}-\lim_n \mathcal{U}_n$ extends $\mathcal{F}$ and these extensions are pairwise distinct by Lemma \ref{lem:rudin}. Hence, $[\mathcal{F}]$ is infinite.

It remains to check that $\mathcal{F}$ is a $Q$-filter. Since $\mathcal{F}$ extends $\mathcal{H}$, it suffices to show that $\mathcal{F}$ contains a selector for $I_{\alpha}$ for any $\alpha < \mathfrak{u}$. By construction, $\langle f_{\alpha}(\mathcal{U}_n): n\in A_{\alpha}\rangle$ is discrete and $A_\alpha \in \mathcal{G}$. For each $n\in A_\alpha$, find $X_n\in f_{\alpha}(\mathcal{U}_n)$ such that the $X_n$ are pairwise disjoint. Also, let $S_n\in\mathcal{U}_n$ be a selector for $I_{\alpha}$. It follows that
\[\bigcup_{n\in A_\alpha} S_n\cap f_{\alpha}^{-1}(X_n)\in \mathcal{F}\]
is a selector for $I_{\alpha}$.
\end{proof}

\pagebreak

Note that essentially the same argument proves item (iii).

\begin{prop}
Assume that there is an $\omega^{\ast}$-discrete ultrafilter of character $>\mathfrak{d}$ and that there are infinitely many $Q$-points. Then there is a $Q$-filter $\mathcal{F}$ such that $[\mathcal{F}]$ is infinite.
\end{prop}

\begin{proof}
Let $\langle \mathcal{U}_n: n \in \omega \rangle$ be a discrete sequence of pairwise non-isomorphic $Q$-points and $\mathcal{W}$ an $\omega^{\ast}$-discrete ultrafilter with $\chi(\mathcal{W})>\mathfrak{d}$. Fix a $Q$-filter-test family $\mathcal{I}$ over $\bigcap_{n \in \omega}\mathcal{U}_n$ of cardinality $\mathfrak{d}$. By induction on $\mathfrak{d}$, construct the filter $\mathcal{G}\subseteq \mathcal{W}$ as in the proof of Theorem \ref{thm:ulessd}. Again, $\mathcal{G}-\lim_n \mathcal{U}_n$ serves as the required $\mathcal{F}$.
\end{proof}

Finally, we prove item (iv).

\begin{prop}
Assume that there are infinitely many pairwise non-isomorphic Ramsey ultrafilters. Then there is a $Q$-filter $\mathcal{F}$ such that $[\mathcal{F}]$ is infinite.
\end{prop}

\begin{proof}
Let $\langle \mathcal{R}_n: n \in \omega\rangle$ be a sequence of pairwise non-isomorphic Ramsey ultrafilters and define $\mathcal{F}:=\bigcap_{n \in \omega}\mathcal{R}_n$. It suffices to verify that $\mathcal{F}$ contains a selector for any interval partition $I$. Let $f$ be the finite-to-one function that is constantly equal to $n$ on the $n$'th interval of $I$. Since each $f(\mathcal{R}_n)$ is a $P$-point, it is not a limit point of any countable set of ultrafilters. It follows that the sequence $\langle f(\mathcal{R}_n): n \in \omega\rangle$ is discrete, and we find sets $X_n \in f(\mathcal{R}_n)$ that are pairwise disjoint. Let $S_n\in \mathcal{R}_n$ be a selector for $I$. Now, the set $\bigcup_{n \in \omega}S_n \cap f^{-1}(X_n)$ is a selector for $I$ and an element of $\mathcal{F}$.
\end{proof}

In the following two subsections, which are slightly tangential to the rest of the paper, we show how the results above strengthen recent results in the literature. 

\subsection{$Q$-filters and non-atomic $Q$-measures}

We remark that the existence of a $Q$-filter $\mathcal{F}$ that extends to infinitely many ultrafilters is equivalent to the existence of a non-atomic $Q$-measure. This improves a result of Borodulin--Nadzieja, Mart{\'\i}nez-Celis, Morawski and {\'S}wierczy{\'n}ska \cite{bdmcms-measures}, who showed one direction of this equivalence (Proposition 4.3) and proved that the existence of infinitely many Ramsey ultrafilters implies the existence of a non-atomic $Q$-measure (Theorem 4.29). However, all the ideas are already in that paper.

\begin{defn}
Let $\mu$ be a finitely additive probabilty measure on $\omega$.
\begin{itemize}
\item $\mu$ is \emph{non-atomic} if for every $\varepsilon>0$, there exists $n \in \omega$ and a partition $\{P_i: i \in n\}$ of $\omega$ such that $\forall i \in n: \mu(P_i)<\varepsilon$.
\item $\mu$ is a \emph{$Q$-measure} if for every interval partition $I$ of $\omega$, there exists a selector $A \subseteq \omega$ for $I$ with $\mu(A)=1$.
\end{itemize}
\end{defn}

\begin{prop}
There is a non-atomic $Q$-measure if and only if there is a $Q$-filter $\mathcal{F}$ such that $[\mathcal{F}]$ is infinite.
\end{prop}

\begin{proof}
If $\mu$ is a non-atomic $Q$-measure, then $\mathcal{F}=\{A \subseteq \omega: \mu(A)=1\}$ is as required (\cite{bdmcms-measures}, Proposition 4.3).

For the other direction, assume that $\mathcal{F}$ is a $Q$-filter such that $[\mathcal{F}]$ is infinite. Let $\langle \mathcal{U}_n: n \in \omega \rangle$ be a discrete sequence in $[\mathcal{F}]$ and let $\mu$ be a measure such that, for any $A \subseteq \omega$, $\mu(A)$ is a limit point of the set
\[\left\{\frac{1}{k}\;(\delta_{\mathcal{U}_0}(A)+\delta_{\mathcal{U}_1}(A)+ \ldots + \delta_{\mathcal{U}_{k-1}}(A)): k \in \omega \setminus \{0\}\right\} \subseteq [0,1],\]
where $\delta_{\mathcal{U}_i}$ is the $0{-}1$ measure induced by $\mathcal{U}_i$. In other words, we choose $\mu$ in the weak$^{\ast}$ closure of the set of measure $\{1/k \;(\delta_{\mathcal{U}_0}+\delta_{\mathcal{U}_1}+ \ldots + \delta_{\mathcal{U}_{k-1}}): k \in \omega \setminus \{0\}\}$. Note that any $A \in \mathcal{F}$ has measure $1$.

It is clear that $\mu$ is a $Q$-measure. To see that it is non-atomic, let $m \in \omega \setminus \{0\}$. Since $\langle \mathcal{U}_n: n \in \omega \rangle$ is discrete, we find pairwise disjoint sets $X_n,\, n \in \omega$, with $X_n \in \mathcal{U}_n$. For $i<m$, let $P_i:=\bigcup \{X_n: n \equiv i \operatorname{mod} m\}$. Now, $\{P_i: i <m\}$ is a partition of $\omega$ and $\forall i <m: \mu(P_i)<1/m$.
\end{proof}

\begin{cor}
Suppose one of the assumptions of (1), (3), (4), or (5) of the Main Theorem, or of (6) holds. If there are infinitely many $Q$-points, then there is a non-atomic $Q$-measure.
\end{cor}

This partially answers Question 7.4 in \cite{bdmcms-measures}, where the authors asked whether the existence of infinitely many $Q$-points implies the existence of a non-atomic $Q$-measure.

\subsection{Infinite closed sets of $Q$-points and Tukey-top $Q$-points}

Addressing a question asked by Benhamou and Dobrinen~\cite{bendob-ufsovermeasurable}, Benhamou and Wu~\cite{benhamouwu-diamond} proved that Tukey-top $Q$-points exist either after adding sufficiently many Cohen reals, or under \textsf{CH}. Later, Raghavan~\cite{raghavan2025q} proved that if $\mathfrak{d}=\aleph_1$ or if there are infinitely many Ramsey ultrafilters, then there exist $2^{\mathfrak{c}}$-many pairwise $\leq_{\operatorname{RK}}$-incomparable Tukey-top $Q$-points. In this subsection, we remark that this is true under the more general assumption that there exists an infinite closed set of $Q$-points, which follows from the existence of infinitely many $Q$-points under any of the assumptions of (1), (3), (4), or (5) of the Main Theorem, or of (6).

Recall that the Tukey order on $\omega^*$ is characterized as follows. Let $\mathcal{U}$ and $\mathcal{V}$ be ultrafilters. $\mathcal{U}$ is \emph{Tukey below} $\mathcal{V}$, denoted $\mathcal{U}\leq_{\operatorname{T}}\mathcal{V}$, if there exists a monotone\footnote{Here, monotone means $A\subseteq B\Rightarrow F(A)\subseteq F(B)$, for all $A,B\in\mathcal{V}$.} map $F: \mathcal{V}\to \mathcal{U}$ such that $\forall A \in \mathcal{U}\,\exists B \in \mathcal{V}: F(B)\subseteq A$. An ultrafilter $\mathcal{U}$ is called \emph{Tukey-top} if $\mathcal{U}$ is Tukey above every directed poset of cardinality $\leq\mathfrak{c}$. For a detailed investigation of the Tukey-types of ultrafilters on $\omega$ we refer the reader to \cite{dobrinen2011tukey} and \cite{raghavan2012cofinal}. Note in particular that $\mathcal{U}\leq_{\operatorname{RK}}\mathcal{V}\implies \mathcal{U}\leq_{\operatorname{T}}\mathcal{V}$.

\begin{fact}\label{fact:cofinalrk}
Any infinite closed subset $[\mathcal{F}]$ of $\omega^{\ast}$ is cofinal in the RK-ordering. In fact, there exists a fixed $f\in {^\omega}\omega$ such that for any ultrafilter $\mathcal{U}$, there exists $\mathcal{V}\in [\mathcal{F}]$ with $\mathcal{U}\leq_{\operatorname{RK}}\mathcal{V}$.
\end{fact}

\begin{proof}
Since $[\mathcal{F}]$ is infinite, we find pairwise disjoint sets $X_n \in [\omega]^\omega, \,n\in \omega,$ such that each $X_n$ has infinite intersection with every element of $\mathcal{F}$. Let $f\in {^\omega}\omega$ be such $X_n=f^{-1}(\{n\})$ for each $n\in \omega$. Now, for any ultrafilter $\mathcal{U}$, $f$ witnesses that every ultrafilter extending $\mathcal{F}\cup \{f^{-1}(A): A \in \mathcal{U}\}$ is RK-above $\mathcal{U}$.
\end{proof}
The following can now be obtained by the preceding fact.
\begin{prop}
Any infinite closed subset $[\mathcal{F}]$ of $\omega^{\ast}$ contains $2^{\mathfrak{c}}$-many pairwise $\leq_{RK}$-incomparable Tukey-top ultrafilters.
\end{prop}

\begin{proof}
By a result of Isbell~\cite{isbell1965category} and Juh\'asz~\cite{juhasz-tukey}, there are $2^{\mathfrak{c}}$-many Tukey-top ultrafilters, say $\mathcal{U}_\alpha, \alpha \in 2^\mathfrak{c}$. By the previous fact, we find $f \in {^\omega}\omega$ and for each $\alpha$ some $\mathcal{V}_\alpha \in [\mathcal{F}]$ such that $f$ witnesses $\mathcal{U}_\alpha \leq_{\operatorname{RK}}\mathcal{V}_\alpha$, and hence $\mathcal{U}_\alpha \leq_{\operatorname{T}}\mathcal{V}_\alpha$. Since $f$ is fixed, it follows that the $\mathcal{V}_\alpha$ are pairwise distinct Tukey-top ultrafilters. Finally, we find a set $\mathcal{X}\subseteq \{\mathcal{V}_{\alpha} :\alpha<2^{\mathfrak{c}}\}$, of cardinality $2^{\mathfrak{c}}$, consisting of pairwise $\leq_{\operatorname{RK}}$-incomparable ultrafilters via an application of Hajnal's free set theorem~\cite{hajnal-freeset} (see also \cite{juhasz-cardfunctionsintop}). We refer the reader to, e.g., \cite{shelahrudin, dobrinen2011tukey, raghavan2025q} for similar applications of Hajnal's theorem.
\end{proof}

\begin{rem}
Instead of using Fact \ref{fact:cofinalrk}, we could have alternatively applied the classical construction of Isbell and Juh\'asz of a Tukey-top ultrafilter using an independent family to the pairwise disjoint sets $\{X_n: n\in\omega\}$ from the proof of Fact \ref{fact:cofinalrk}.
\end{rem}
\begin{cor}
Suppose one of the assumptions of (1), (3), (4), or (5) of the Main Theorem, or of (6) holds. If there are infinitely many $Q$-points, then there are $2^{\mathfrak{c}}$-many pairwise $\leq_{\operatorname{RK}}$-incomparable Tukey-top $Q$-points. 
\end{cor}

\section{The number of $Q$-points under $\mathfrak{u}=\mathfrak{d}$}

It remains to prove item (2) of the Main Theorem. By Corollary \ref{cor:pigeonhole}, it suffices to show that under $\mathfrak{u}=\mathfrak{d}$, the existence of infinitely many $Q$-points implies the existence of at least $2^{\mathfrak{d}}$-many $Q$-points. Unlike in the previous proofs, we will have to make sure that these $Q$-points we construct are pairwise non-isomorphic, since we cannot appeal to Fact \ref{fact:numberofiso}.

\begin{thm}\label{thm:ugeqd}
Assume that $\mathfrak{u}\geq \mathfrak{d}$ and that there are infinitely many $Q$-points. Then the number of $Q$-points is at least $2^{\mathfrak{d}}$.
\end{thm}

We need the following lemma, which is a consequence of Lemma 7 of Banakh and Blass \cite{banakh2006number}. We give a proof for convenience.

\begin{lem}\label{lem:ketonen}
Let $\{ \mathcal{U}_n: n \in \omega \}$ be a set of pairwise non-nearly coherent ultrafilters and $\mathcal{F}$ a filter with $\chi(\mathcal{F})<\mathfrak{d}$. There exists $A \in [\omega]^\omega$ intersecting every element of $\mathcal{F}$ such that $\{\mathcal{U}_n: n \in A\}$ is discrete, with at most one limit point in $\{\mathcal{U}_n: n \in \omega\}$.
\end{lem}

\begin{proof}
Let $\mathcal{V}$ be any ultrafilter extending $\mathcal{F}$. For all but at most one $m \in \omega$, we have $\mathcal{V}-\lim_n \mathcal{U}_n \neq \mathcal{U}_m$, and hence find $X_m \in \mathcal{U}_m$ such that $\omega \setminus X_m \in \mathcal{V}-\lim_n \mathcal{U}_n$. Let $A_m := \{n \in \omega: \omega \setminus X_m \in \mathcal{U}_n\}\in \mathcal{V}$. By a lemma due to Ketonen \cite{ketonen1976existence}, using the assumption $\chi(\mathcal{F})<\mathfrak{d}$, there exists a pseudo-intersection $A$ of the $A_m$ that intersects every element of $\mathcal{F}$.

Now, for each $m$ as above, all but finitely many of the $\mathcal{U}_n$ for $n \in A$ do not contain $X_m \in \mathcal{U}_m$, which shows that $\mathcal{U}_m$ is not a limit point of $\{\mathcal{U}_n: n \in A\}$. By possibly removing the single bad $m$ from $A$, we are done.
\end{proof}

\begin{proof}[Proof of Theorem \ref{thm:ugeqd}]
Assume that $\mathcal{U}_n$ for $n \in \omega$ are pairwise non-isomorphic $Q$-points. Let $\mathcal{I}$ be a $Q$-test family and $\mathcal{T}$ a test family of finite-to-one functions, both over the Fr\'echet filter and both of cardinality $\mathfrak{d}$. For each interval $I \in \mathcal{I}$, let $g_I$ be the finite-to-one function that is constantly equal to $n$ on the $n$-th interval of $I$. Let $\{f_\alpha: \alpha < \mathfrak{d}\}$ be an enumeration of $\mathcal{T}\cup \{g_I: I \in \mathcal{I}\}$. By induction on $\mathfrak{d}$, we construct for each $\varphi \in {^\mathfrak{d}}2$ an ultrafilter $\mathcal{V}_\varphi$ such that the ultrafilters $\mathcal{V}_\varphi-\lim_n \mathcal{U}_n$ are pairwise non-isomorphic $Q$-points.

At stage $\alpha+1 \in \mathfrak{d}$, assume that for each $\eta \in {^\alpha}2$ a filter $F_\eta$ with $\chi(F_\eta)<\mathfrak{d}$ is given. Applying Lemma \ref{lem:ketonen}, we find $A \in [\omega]^\omega$ intersecting every element of $\mathcal{F}_\eta$ such that $\{f_\alpha(\mathcal{U}_n): n \in A\}$ is discrete with at most one limit point among all the $f_\alpha(\mathcal{U}_n),\; n \in \omega$. Since $\mathcal{F}_\eta$ is $<\mathfrak{u}$-generated, there is a decomposition of $A_\eta$ into disjoint sets $\A_\eta^0$ and $A_\eta^1$ that both intersect every element of $\mathcal{F}_\eta$. Add $A_\eta^i$ to $\mathcal{F}_\eta$ to obtain distinct $F_{\eta^{\smallfrown}0}$ and $F_{\eta^{\smallfrown}1}$.

At limit stages, we take unions of the previously constructed filters. Finally, let $\mathcal{V}_\varphi$ be any ultrafilter extending $\bigcup_{\alpha \in \mathfrak{d}}\mathcal{F}_{\varphi \restriction \alpha}$. By construction, the $\mathcal{V}_\varphi$ are pairwise distinct.

By the same argument as in the proof of Theorem \ref{thm:ulessd}, each $\mathcal{V}_{\varphi}-\lim_n \mathcal{U}_n$ is a $Q$-point. It remains to check that $\mathcal{V}_{\varphi}-\lim_n \mathcal{U}_n$ and $\mathcal{V}_{\psi}-\lim_n \mathcal{U}_n$ are non-isomorphic for $\varphi\neq \psi$.

Since $f(\mathcal{V}-\lim_n \mathcal{U}_n)=\mathcal{V}-\lim_n f(\mathcal{U}_n)$ for any $f\in {^\omega}\omega$ and since $\{f_\alpha: \alpha \in \mathfrak{d}\}$ is a test family, it suffices to show that \begin{equation}\label{eq:limits}
\mathcal{V}_{\varphi}-\lim_n f_\alpha(\mathcal{U}_n) \neq \mathcal{V}_{\psi}-\lim_n f_\alpha(\mathcal{U}_n),\;\text{ for all $\alpha<\mathfrak{d}$.}
\end{equation}
We find $A \in \mathcal{V}_\varphi$ and $B \in \mathcal{V}_\psi$ such that $\{f_\alpha(\mathcal{U}_n): n \in A\}$ and $\{f_\alpha(\mathcal{U}_n): n \in B\}$ are discrete, with at most one limit point among $\{f_\alpha(\mathcal{U}_n): n \in \omega\}$. Hence, by removing at most one element from $A$ and at most one element from $B$, we may assume that $\{f_\alpha(\mathcal{U}_n): n \in A \cup B\}$ is discrete. Since $A \cup B \in \mathcal{V}_{\varphi}\cap \mathcal{V}_\psi$, equation (1) follows from Lemma \ref{lem:rudin}.

\end{proof}

\section{Closing Remarks and Questions}\label{sec:questions}

As mentioned in the introduction, we do not know the answer to the following question.

\begin{question}
Is there a model with exactly $2^{\mathfrak{d}}$-many $Q$-points and $2^{\mathfrak{d}}<2^{\mathfrak{c}}$?
\end{question}
It follows from our main theorem that such a model would need to satsify the following.
\begin{enumerate}[label=(\roman*)]
\item $\operatorname{cov}(\mathcal{M})<\mathfrak{u}=\mathfrak{d}<\mathfrak{c}$,
\item every $Q$-point is $(\mathfrak{u}=\mathfrak{d})$-generated,
\item every $\omega^{\ast}$-discrete ultrafilter is $(\mathfrak{u}=\mathfrak{d})$-generated. In particular, every $P$-point is $(\mathfrak{u}=\mathfrak{d})$-generated.
\item there are at most finitely many Ramsey ultrafilters.
\end{enumerate}





Our second question concerns the relationship between $Q$-filters that extend to infinitely many ultrafilters and infinite closed sets of $Q$-points. In particular,

\begin{question}
Is the existence of an infinite closed set of $Q$-points equivalent to the existence of a $Q$-filter $\mathcal{F}$ such that $[\mathcal{F}]$ is infinite?
\end{question}

It is not hard to verify that a closed set $[\mathcal{G}]$ in $\omega^{\ast}$ consists entirely of $Q$-points if and only if $\mathcal{G}$ has the following property: For any interval partition $I$ of $\omega$, there exists some $A \in \mathcal{G}$ and some $n \in \omega$ such that $A$ intersects every interval in $I$ in at most $n$ points.

\bibliographystyle{plain}
\bibliography{bib}

\end{document}